\documentclass[11pt]{amsart}

\usepackage[singlespacing]{setspace} 

\usepackage{amsfonts}
\usepackage{mathrsfs} 
\usepackage{dsfont}  

\usepackage[active]{srcltx}
\usepackage{hyperref}

\usepackage[initials]{amsrefs}  
\usepackage{bm,amsmath,amssymb}

\usepackage[all]{xy}
\usepackage{stmaryrd}

\usepackage[utf8]{inputenc}

\renewcommand{\le}{\leqslant}
\renewcommand{\ge}{\geqslant}

\newcommand{\Lquote}[1]{``#1"}

\newcommand {\pnorm}[1]   {\left\lVert #1 \right\rVert}
\newcommand {\pabs}[1]    {\left\lvert #1 \right\rvert}

\newcommand{\Mdemi}{\frac{1}{2}}

\newcommand{\qtext}[1]{\quad\text{#1}}

\newcommand{\SB}{\backslash}

\newcommand{\Tab}{\text{ab}}

\DeclareMathOperator {\Mrec}{rec}

\DeclareMathOperator {\Msgn}  {sgn}
\DeclareMathOperator {\Mrang}{rang}

\DeclareMathOperator {\Mdim} {dim}

\DeclareMathOperator {\Mvol} {vol}
\DeclareMathOperator {\MGal} {Gal}
\DeclareMathOperator {\MCl}  {Cl}

\DeclareMathOperator {\MEnd} {End}

\DeclareMathOperator {\MPic} {Pic}

\DeclareMathOperator {\Mim}  {\Im m}

\DeclareMathOperator {\MRes} {Res}

\newcommand{\Mdede}[4]{
\begin{pmatrix}
#1&#2  \\
#3&#4 \\
\end{pmatrix}
}

\usepackage{dsfont}  
\usepackage{mathrsfs} 

\newcommand{\BmA}{\mathbb{A}}

\newcommand{\BmC}{\mathbb{C}}
\newcommand{\BmD}{\mathbb{D}}

\newcommand{\BmG}{\mathbb{G}}

\newcommand{\BmP}{\mathbb{P}}
\newcommand{\BmQ}{\mathbb{Q}}
\newcommand{\BmR}{\mathbb{R}}
\newcommand{\BmS}{\mathbb{S}}

\newcommand{\BmZ}{\mathbb{Z}}

\newcommand{\CmH}{\mathcal{H}}
\newcommand{\CmI}{\mathcal{I}}

\newcommand{\CmL}{\mathcal{L}}

\newcommand{\CmN}{\mathcal{N}}
\newcommand{\CmO}{\mathcal{O}}

\newcommand{\CmR}{\mathcal{R}}
\newcommand{\CmS}{\mathcal{S}}

\newcommand{\SmA}{\mathscr{A}}

\newcommand{\FmH}{\mathfrak{H}}

\theoremstyle{plain}
\newtheorem{theorem}{Théorème}
\newtheorem{proposition}{Proposition}
\newtheorem{lemma}{Lemme}
\newtheorem{definition}{Définition}

\theoremstyle{remark}
\newtheorem{remark}{Remarque}
\newtheorem{example}{Exemple}

\usepackage{longtable}    

\newcommand{\Mclgr}{\widehat{\MCl}}

\DeclareMathOperator{\Mht}{ht}
\newtheorem{hypothese}{Condition}

\begin{document}
\title[]{Sur le rang des courbes elliptiques sur les corps de classes de Hilbert}
\author[]{Nicolas Templier}
%
%
\address{Institute for Advanced Study, Princeton NJ-08540}
\email{ntemplier@ias.edu}

%
\dedicatory{\Lquote{ces brouilleries inexplicables}}
\date{\today}
\maketitle
\begin{abstract} Let $E/\BmQ$ be an elliptic curve and $D<0$ a sufficiently large fundamental discriminant. If $E(\overline{\BmQ})$ contains Heegner points of discriminant $D$, those points generate a subgroup of rank at least $\pabs{D}^{\delta}$, where $\delta>0$ is an absolute constant. This result is compatible with the Birch and Swinnerton-Dyer conjecture.\\

\noindent {\bf R\'esum\'e.} Soit $E/\BmQ$ une courbe elliptique. Soit $D<0$ un discriminant fondamental suffisamment grand. Si $E(\overline{\BmQ})$ contient des points de Heegner de discriminant $D$, ces points engendrent un sous-groupe dont le rang est supérieur à $\pabs{D}^{\delta}$, où $\delta>0$ est une constante absolue. Ce résultat est en accord avec la conjecture de Birch et Swinnerton-Dyer.
\end{abstract}
\tableofcontents
\section{Introduction et énoncé des résultats.}\label{sec:intro}
\subsection{Points algébriques sur les courbes elliptiques.}\label{sec:intro:points}
Soit $E$ une courbe elliptique définie sur le corps des rationnels $\BmQ$. Pour un corps de nombres $M$, le groupe abélien $E(M)$ des points rationnels de $E$ définis sur $M$ est de type fini par le théorème de Mordell-Weil. De nombreux travaux portent sur l'étude de ce groupe et beaucoup de questions demeurent largement ouvertes, notamment la conjecture de Birch et Swinnerton-Dyer.

Un outil fondamental dans l'étude de $E(M)$ est la construction de \Lquote{points spéciaux}. La théorie de la multiplication complexe permet de construire les points de Heegner qui sont définis sur les corps de classes des corps quadratiques imaginaires. Pour $D<0$ discriminant fondamental, on notera $H_D$ le corps de classes de Hilbert de $\BmQ(\sqrt{D})$. Rappelons que la formule de B.~Gross et D.~Zagier~\cite{GZ} et la méthode de descente de V.~Kolyvagin~\cites{Kolyvagin:finiteness,Kolyvagin:Mordell-Weil} font appels à ces points de Heegner pour montrer le résultat suivant. Lorsque l'ordre d'annulation de $L(s,E)$ en $s=1/2$ est au plus $1$, il est égal au rang du groupe $E(\BmQ)$.

Dans cet article, on s'intéressera au groupe $E(H_D)$ et à la question de savoir si les points de Heegner sont linéairement indépendants ou si au contraire ils auraient tendance à \Lquote{s'aligner}. Une autre fa\c{c}on de formuler le problème consiste à demander quelle est la taille relative du sous-groupe de $E(H_D)$ engendré par les points de Heegner.

Considérons l'énoncé quantitatif suivant. Il existe une constante absolue $\delta>0$ telle que
\begin{equation}\label{intro:in}
 \Mrang E(H_D) > |D|^\delta
\end{equation}
lorsque $\pabs{D}$ est suffisamment grand et satisfait une condition de compatibilité. On peut faire les remarques suivantes. D'une part cet énoncé serait conséquence de la conjecture de Birch et Swinnerton-Dyer qui dit essentiellement que $\delta\approx 1/2$ serait admissible (voir \S~\ref{sec:intro:dich}). D'autre part Ph.~Michel et A.~Venkatesh~\cite{MV05} ont montré dans un cas particulier (condition de Heegner) qu'il serait conséquence d'une hypothèse quantitative sur les petits premiers décomposés dans $\BmQ(\sqrt{D})$. Enfin il serait conséquence de certaines heuristiques sur la non-annulation de valeurs critiques de fonctions $L$ comme annoncé dans \cite{cong:ICM06:MV}*{section 2.4}. Étant donné ce faisceau convergent d'indices, on peut donc espérer que l'énoncé soit vrai!

Dans ce papier on démontre l'inégalité~\eqref{intro:in} inconditionnellement en suivant la troisième approche (voir le Théorème~\ref{th:main} pour un énoncé précis). La condition de compatibilité~\eqref{eq:th:main:sgn}, qui est équivalente à l'existence de points de Heegner de discriminants $D$ sur $E$, est essentiellement \emph{optimale}, voir le \S~\ref{sec:intro:dich}. On atteint ce degré de généralité par une nouvelle idée simple qui est suffisamment robuste pour s'adapter à toutes les situations.

Une variante intéressante consiste à considérer des corps de classes ramifiés. Ce cas a été étudié en détail pour son interaction avec la théorie d'Iwasawa et un résultat important est le suivant. En réponse à une conjecture de B.~Mazur, V.~Vatsal~\cite{Vats03}*{Theorem 1.4} et C.~Cornut~\cites{Corn02,Cornut:cras} ont montré que si l'on \emph{fixe} un discriminant $D$ et un nombre premier $p$ et que le corps $M\supset H_D$ est le corps de classes de $\BmQ(\sqrt{D})$ de conducteur $p^n$, le rang de $E(M)$ est $\gg_{D,p} p^n$ lorsque $n$ est assez grand. Ce résultat a été étendu aux corps de classes des corps totalement réels~\cites{CV05, CV05:partII} et au cas où le conducteur de $M$ est concentré en un nombre fini et fixe de places~\cite{AN08} par la même méthode. Le lecteur peut également consulter avec profit les survols~\cites{cong:heeg04:vats,book:Nekovar:asterisque,cong:ICM06:vats}.

Il y a des analogies profondes entre la démonstration du théorème de Cornut-Vatsal et la démonstration du Théorème~\ref{th:main}. On peut trouver une discussion lucide d'un tel rapprochement dans l'introduction de~\cite{MV05}. Mentionnons que la question d'unifier ces deux théorèmes est intéressante et que nous y reviendrons dans un article futur.

Récemment, A.~Buium et B.~Poonen~\cite{BP07} ont démontré un résultat qui a également trait à l'indépendance des points de Heegner sur les courbes elliptiques. Leur résultat admet le corollaire suivant: un sous-groupe \Lquote{de rang fini} $\Gamma\subset E(\overline{\BmQ})$, ne contient pas de point de Heegner de discriminant $D$ lorsque $\pabs{D}$ est suffisamment grand (en fonction de $\Gamma$). Dans la terminologie de~\cite{BP07}, un groupe est de rang fini si la dimension sur $\BmQ$ de $\Gamma \otimes_\BmZ \BmQ$ est finie. Ce corollaire améliore un résultat antérieur de N.~Schapparer et J.~Nekovar~\cite{NS99} qui montrent que lorsque $\pabs{D}$ est suffisamment grand, les points de Heegner de discriminant $D$ ne sont pas de torsion. Il existe également un résultat conditionnel d'indépendance dû à M.~Rosen et J.~Silverman~\cite{RS07} qui est de nature différente.

Rappelons finalement que les travaux \cites{BFH90,Murty90:mean,Iwan90:vanishing} montrent qu'il existe une infinité de discriminants fondamentaux $D$ tels que $E(\BmQ(\sqrt{D}))$ soit infini, et que ce fait constitue une étape dans la démonstration du théorème de Kolyvagin mentionné plus haut.

\subsection{Une dichotomie.}\label{sec:intro:dich} Soit $E/\BmQ$ une courbe elliptique et $D$ un discriminant fondamental. La conjecture de Birch et Swinnerton-Dyer (BSD) admet des conséquences profondes quant au rang de $E(H_D)$. On va voir qu'asymptotiquement deux cas bien distincts devraient se produire. 

On note $\MCl_D$ le groupe des classes d'idéaux de $\BmQ(\sqrt{D})$, et $\Mclgr_D$ son dual. Le cardinal de $\MCl_D$ est le nombre de classes $h(D)$. On utilisera dans ce paragraphe librement certaines propriétés connues des fonctions $L$, le lecteur peut trouver plus de détails dans la suite et dans le survol~\cite{book:darm04}.

Comme l'extension $H_D/\BmQ$ est \Lquote{dihédrale} (au sens où son groupe de Galois est une extension de $\BmZ/2\BmZ$ par le groupe abélien $\MCl_D$), la fonction $L$ de $E$ sur $H_D$ se factorise en:
\begin{equation}
  L(s,E\otimes_\BmQ H_D)=
  \prod_{\chi\in \Mclgr_D}
  L(s,E\times \chi).
\end{equation}

Ici $E\times \chi$ est une représentation automorphe autoduale que l'on peut construire par convolution de Rankin-Selberg. Dans tout le texte on fait l'hypothèse simplificatrice suivante:
\begin{equation}\label{S}\tag{S}
\begin{aligned}
 \parbox{7cm}{\emph{Le signe de l'équation fonctionnelle de $L(s,E\times \chi)$ est indépendant de $\chi\in\Mclgr_D$.}}
\end{aligned}
\end{equation}
Cette hypothèse est relativement faible, et vérifiée dans la plupart des cas. Toutefois on met en garde le lecteur sur le fait qu'elle n'est pas toujours vérifiée.

Alors ce signe ne dépend que de $E$ et de $D$, on le note $\Msgn(E,D)\in\{\pm 1\}$. C'est par exemple le signe de l'équation fonctionnelle de la fonction $L $ du changement de base $E\otimes_\BmQ \BmQ(\sqrt{D})$. Selon la valeur de $\Msgn(E,D)$ on a deux situations bien distinctes:

{\bf $1^{\text{er}}$ cas.} Lorsque $\Msgn(E,D)=-1$, il est clair que 
 $L(1/2,E\times \chi)=0$ pour tout $\chi \in \Mclgr_D$. Ainsi l'ordre d'annulation de $L(s,E\otimes_\BmQ H_D)$ est au moins $h(D)$. Admettant la conjecture BSD on aurait donc:
\begin{equation}
 \Mrang E(H_D) \ge h(D) \gg_\epsilon \pabs{D}^{1/2-\epsilon}, \qtext{pour tout $\epsilon>0$.}
\end{equation}
La deuxième inégalité est le théorème de Siegel~\cite{Sieg:classnumber}.

{\bf $2^{\text{nd}}$ cas.} Lorsque $\Msgn(E,D)=+1$ on ne peut pas tirer de conclusion immédiate quant à l'annulation de $L(1/2,E\times \chi)$. Les conjectures de type Katz-Sarnak~\cite{book:KS} suggèrent que la fonction $L(s,E\times \chi)$ ne s'annule pas \Lquote{en général} au point critique $1/2$. Il serait long de préciser quantitativement cette heuristique (formalisme des petits zéros de familles de fonctions $L$). Il est difficile de formuler une conjecture précise (borné? croissance au plus logarithmique?), mais on s'attend clairement à ce que le rang de $E(H_D)$ soit significativement plus petit que dans le $1^{\text{er}}$ cas.

\subsection{Résultat principal.}\label{sec:intro:resultat} Dans cet article on démontre le résultat quantitatif suivant qui va dans la direction de la conjecture de Birch et Swinnerton-Dyer en établissant une version faible de la conclusion du $1^{\text{er}}$ cas ci-dessus. \'Etant donnée la dichotomie, la condition~\eqref{eq:th:main:sgn} est optimale.
\begin{theorem}\label{th:rang}\label{th:main}
 Il existe un réel $\delta >0$ tel que l'on ait, pour toute courbe elliptique $E$ définie sur $\BmQ$, l'inégalité
  \begin{equation}\label{principal}
    \Mrang E(H_D) > |D|^\delta
  \end{equation}
quel que soit le discriminant fondamental négatif $D$ qui est assez grand et qui satisfait~\eqref{S} et
\begin{equation}\label{eq:th:main:sgn}
 \Msgn (E,D)=-1.
\end{equation}
\end{theorem}

\begin{remark} Soit $N>1$ le conducteur de $E$. Il est possible de préciser le fait que $D$ est assez grand en disant que $|D|\ge N^A$ où $A>0$ est une constante absolue, mais ineffective. Pour cela, il suffit de vérifier que toutes les constantes de la démonstration ont une dépendance polynomiale en $N$, ce qui ne pose pas de difficulté particulière.
\end{remark}

\begin{remark} Toutes les notions de cet article peuvent être vues à isogénie près. Par exemple, ne dépendent que de la classe d'isogénie de $E$: le rang $E(H_D)$, la construction des points de Heegner, les fonctions $L(s,f\times \chi)$, donc en particulier le signe $\Msgn(E,D)$, le conducteur $N$. Cela suggère donc que la démonstration aussi doit être suffisament robuste pour ne pas être sensible aux isogénies.
\end{remark}

\begin{remark}\label{rem:lindelof} L'exposant $\delta$ provient d'une majoration de sous-convexité pour la convolution de Rankin-Selberg $GL(2)_\BmQ \times GL(2)_\BmQ$. Dans le contexte du Théorème~\ref{th:main}, cette majoration profonde a été établie par Ph.~Michel dans~\cite{Mich04} avec l'exposant $\delta=1/1057$. Sous l'hypothèse de Lindel\"of généralisée, on pourrait choisir $\delta=1/2-\epsilon$ avec $\epsilon$ arbitrairement petit.
\end{remark}

\begin{remark}
Dans \cite{Temp:cras}, l'auteur a énoncé un résultat similaire où~\eqref{eq:th:main:sgn} est remplacée par la \Lquote{condition de Heegner}. La condition de Heegner est le fait que tous les facteurs premiers de $N$ sont décomposés dans $\BmQ(\sqrt{D})$. La condition~\eqref{eq:th:main:sgn} est optimale au sens où c'est la condition la plus faible sous laquelle on peut construire des points de Heegner de discriminant $D$ sur $E$.
\end{remark}

\begin{remark} 
G.~Ricotta et T.~Vidick~\cites{RV05,RT08} ont réalisé des calculs numériques et établi que l'inégalité~\eqref{principal} est valide en moyenne sur $D$ sous la condition de Heegner.
\end{remark}

\begin{remark} 
Dans~\cite{MV05}*{Theorem 3}, on peut trouver une démonstration conditionnelle du Théorème~\ref{th:main}. La démonstration repose sur une hypothèse délicate notée $\CmS_{\beta,\theta}$. Sous l'hypothèse de Lindelöf généralisée, $\CmS_{\beta,\theta}$ est vraie pour tout $0<\beta <1$ et $\theta>0$ ce qui impliquerait~\eqref{principal} avec l'exposant $\delta=1/10$. Cet exposant est moins bon que l'exposant $1/2$ discuté dans la Remarque~\ref{rem:lindelof}, mais il faut souligner que le raisonnement de~\cite{MV05}*{section 4} ne fait pas appel à la profonde formule de Gross-Zagier.
\end{remark}

\begin{remark}
La signification de l'hypothèse $\CmS_{\beta,\theta}$, est qu'il existe \Lquote{beaucoup} de petits premiers décomposés dans $\BmQ(\sqrt{D})$ ce qui s'interprète géométriquement en le fait qu'il existe \Lquote{beaucoup} de points de Heegner qui sont hauts dans la pointe de $X_0(N)$ -- au sens où la partie imaginaire est grande. Dans le chapitre~\ref{sec:Heegner}, on démontre le Théorème~\ref{th:main} sous la condition de Heegner et notre démonstration repose également sur un résultat quantitatif (les Lemmes~\ref{lem:LD:minoration} et \ref{lem:heegnerpointe}). On peut comprendre le raisonnement de la manière suivante: la majoration de Burgess implique que les points de Heegner sont hauts dans la pointe \emph{en moyenne} (c'est une version faible du théorème d'équidistribution de Duke).
\end{remark}

\begin{remark}
Jusqu'à présent, le seul résultat connu concernant le rang de $E(H_D)$ est le fait qu'il est $\ge 1$ pour $\pabs{D}$ assez grand. En effet par le résultat de Nekovar-Schapparer~\cite{NS99}, les points de Heegner de discriminant $D$ ne sont pas de torsion pour $\pabs{D}$ assez grand. Signalons qu'il y a un point commun entre~\cite{NS99} et notre approche: une étape importante de la démonstration du Théorème~\ref{th:rang} consiste à établir (voir~\eqref{minoration:heegner}), que la hauteur de Néron-Tate des points de Heegner tend vers $+\infty$ avec $D$ (rappelons que les points de torsion sont les points de hauteur nulle).
\end{remark}

\begin{remark}
 Bien que les travaux de Buium-Poonen~\cite{BP07} et le présent article établissent des résultats de nature différente et ont été conduits indépendamment, il est intéressant de noter que certains arguments sont similaires, notamment l'idée de comparer deux *-limites de mesures et de montrer qu'elles sont distinctes (c'est aussi l'argument clé dans la démonstration de la conjecture de Bogomolov~\cites{Zhan98,Ullm98}).
\end{remark}

\subsection{Le plan de la démonstration.}\label{sec:intro:plan}
Dans ce paragraphe nous expliquons notre stratégie pour démontrer le Théorème~\ref{th:main}. 

La première observation est que la condition~\eqref{eq:th:main:sgn} est \emph{équivalente} à l'existence de points de Heegner de discriminant $D$ sur $E(H_D)$. Ce fait classique est bien entendu une clé pour obtenir le Théorème~\ref{th:main} dans ce degré optimal de généralité.~\footnote{Dit autrement: dès que le le groupe $E(H_D)$ est gros ($1^{\text{er}}$ cas), on~\Lquote{sait} construire des points algébriques.} Notre objectif est donc de montrer que ces points sont suffisamment indépendants. Précisément on montrera que le sous-groupe engendré par les points de Heegner de discriminant $D$ est de rang $>\pabs{D}^\delta$ pour $D$ assez grand. 

Pour cela on utilise l'action par le groupe de Galois $\MGal(H_D/\BmQ(\sqrt{D}))=\MCl_D$ et la formule de Gross et Zagier dont on rappelle en détail l'énoncé dans le chapitre~\ref{sec:GZ}. Soit $\chi$ un caractère de $\MCl_D$. On déduit de cette formule que l'espace $\chi$-isotypique $(E(H_D)\otimes_\BmZ \BmC)^\chi$ est non réduit à $\{0\}$ si la valeur spéciale de la dérivée $L'(\Mdemi,E\times \chi)$ est non nulle.

Pour étudier la non-annulation de $L'(\Mdemi,E\times \chi)$, on fait une moyenne sur les caractères $\chi$, voir l'équation~\eqref{GZ} (on appellera cette moyenne \Lquote{moment}). On a besoin d'une estimée asymptotique pour ce moment, ou plus exactement d'une minoration. Pour cela, on utilise \emph{une seconde fois} la formule de Gross et Zagier: le moment est égal à la hauteur de Néron-Tate d'un point de Heegner de discriminant $D$ sur $E$.

La minoration de cette hauteur (Proposition~\ref{prop:height}) est le point central de la démonstration et fait l'objet du chapitre~\ref{sec:general}. Il est connu des experts (bien que non publié), que la hauteur des points spéciaux ou plus généralement des sous-variétés spéciales des variétés de Shimura ne s'accumule pas en zéro. La Proposition~\ref{prop:height} dit que c'est également le cas lorsque l'on projette sur les variétés abéliennes quotientes. On montrera ce résultat dans le chapitre~\ref{sec:general} en faisant appel à deux résultats profonds d'équidistribution en géométrie arithmétique: le théorème de L.~Szpiro, E.~Ullmo et Zhang~\cite{SUZ97} concernant les petits points et le théorème de W.~Duke~\cite{Duke88} concernant les points de Heegner.

Pour conclure à la non-annulation de $L'(\Mdemi,E\times\chi)$, on fait appel à la majoration de sous-convexité~\eqref{subconvex} démontrée par Michel~\cite{Mich04}.

\begin{remark}
Une différence notable avec les travaux de Cornut-Vatsal est la suivante. Dans les travaux de Cornut-Vatsal, il est démontré que la dérivée spéciale $L'(1/2,E\times \chi)$ ne s'annule pas pour \emph{un} caractère $\chi$ suffisamment ramifié. Puis le résultat d'algébricité de Shimura implique que la non-annulation se produit pour tous les conjugués de $\chi$ par Galois (cette idée a été exploitée pour la première fois dans les travaux de Rohrlich~\cites{MR82,Rohr80a,Rohr80b,Rohr80c} -- qui contiennent aussi l'idée de former le moment d'ordre un). Si l'on fixe $D$ et $p$ (comme le font Cornut-Vasal), le nombre d'orbites par Galois des caractères anticyclotomiques de conducteur $p$ est borné lorsque $n\to \infty$. Cela vient du fait que l'on a des $\BmZ_p$-extensions. Ainsi la non-annulation pour un caractère de la famille (ou plus précisément un nombre fini fixé) entraîne la non-annulation pour toute la famille. Par contre, dans le cadre du Théorème~\ref{th:main} cet argument d'algébricité ne s'applique pas: les caractères du groupe de classes (=non ramifiés) ne sont pas en général conjugués.
\end{remark}

\begin{remark}
 L'auteur a développé d'autres méthodes pour estimer asymptotiquement le moment~\eqref{GZ}, la clé de voûte de la non-annulation. On renvoie à~\cite{Temp:shifted} pour une approche purement analytique. On trouvera également dans~\cite{Temp:shifted} une comparaison précise avec le présent article.
\end{remark}

\subsection{Reformulation.}
Dans la démonstration résumée ci-dessus on peut être plus précis encore et rendre complètement transparent le rôle de la formule de Gross-Zagier vis-à-vis de la conjecture BSD.

Ou bien $L'(1/2,E\times \chi)\not =0$ et alors on peut conclure par la formule de Gross-Zagier que $(E(H_D)\otimes_\BmZ \BmC)^\chi$ est non réduit à $\{0\}$. Ou bien $L'(1/2,E\times \chi) =0$, et la conjecture BSD impliquerait que $(E(H_D)\otimes_\BmZ \BmC)^\chi$ est de dimension au moins $3$.

La démonstration peut donc s'interpréter ainsi. La deuxième alternative est la plus favorable pour le rang, mais sa conclusion étant conjecturale on la met de coté! En fait on concentre même tous nos efforts pour montrer (inconditionnellement) que la première alternative se produit suffisamment souvent!

\subsection{Variétés abéliennes modulaires.}
Dans l'énoncé du Théorème 1, il est possible de remplacer la courbe elliptique $E$ par une variété abélienne modulaire. On explique dans ce paragraphe quelles sont les modifications à apporter pour traiter ce cas. Soit $f$ une forme primitive de poids $2$ et $A$ la variété abélienne associée par la construction d'Eichler-Shimura.

La condition~\eqref{eq:th:main:sgn} est remplacée par le fait que le changement de base de\footnote{Il serait interessant de caractériser ce signe en fonction de $A$ et $D$ seulement, sans faire appel à la forme $f$.} $f$ de $\BmQ$ vers $\BmQ(\sqrt{D})$ a une équation fonctionnelle avec signe $-1$. Le Lemme~\ref{lem:singularitepointe} doit être modifié de la manière suivante: {\it la mesure image $\varphi_*\mu$ est distincte de la mesure de Haar d'un translaté d'une sous-variété abélienne de $A$}. La démonstration de cette variante ne pose pas de difficultés particulières. Le théorème de Szpiro-Ullmo-Zhang~\cite{SUZ97} s'applique aux variétés abéliennes. Le reste de la démonstration s'adapte sans modification.

\begin{remark}
Grâce à la conjecture de Serre~\cite{bour:serre-conj}, on sait que certaines variétés abéliennes sont modulaires. C'est le cas des $\BmQ$-courbes elliptiques au sens de Ribet~\cite{Ribe92}, c'est-à-dire que $E$, définie sur $\overline{\BmQ}$, est isogène à tous ses conjuguées. C'est également le cas des $GL_2$-variétés abéliennes simples (Ribet~\cite{Ribe92} montre qu'une courbe elliptique est une $\BmQ$-courbe si et seulement elle est le quotient d'une $GL_2$-variété abélienne). On dit que $A$ est une $GL_2$-variété abélienne simple lorsque l'algèbre $\MEnd_\BmQ(A)$ est un corps de nombres de degré $\Mdim(A)$.
\end{remark}

\subsection{Problèmes ouverts.}\label{sec:intro:futur}
Comme on l'a dit dans la Remarque~\ref{rem:lindelof}, sous l'hypothèse de Lindelöf on pourrait choisir $\delta$ arbitrairement proche de $1/2$. On connaît des approches analytiques (procédé de mollification) pour démontrer des théorèmes de non-annulation à cette précision, voir notamment~\cites{Blom04,IS00,KMV00}. Dans le cas présent, il faudrait être en mesure d'évaluer asymptotiquement le second moment $\frac{1}{h(D)}\sum_{\chi \in \Mclgr_D} L'(1/2,f\times \chi)^2$ (en fait le second moment tordu pour être précis). Le grand écart entre la taille de la famille (environ $\pabs{D}^{1/2}$) et le conducteur de la famille (qui est $\pabs{D}^4$) rend cette analyse délicate, et en fait hors de portée des techniques actuelles.

En fait il serait très intéressant (et encore plus difficile!) de montrer que \emph{toutes} les dérivées spéciales $L'(1/2,E\times \chi)$ sont non nulles pour $D$ assez grand \emph{et} $\chi \in \Mclgr_D$. On pourrait alors en déduire que les points de Heegner engendrent un sous-groupe de $E(H_D)$ qui est de rang $h(D)$ et d'indice fini, comme cela est signalée à la deuxième page de l'introduction de~\cite{Darm:integration}.

Signalons que la plupart des questions relatives à $E(H_D)$ pour $D>0$ (corps de classes des corps quadratiques réels) sont ouvertes à l'heure actuelle, voir~\cite{Darm:ICM}. Pourtant la conjecture BSD implique de la même fa\c{c}on que $\Mrang E(H_D) \ge h(D)$ lorsque $\Msgn(E,D)=-1$. Bertolini-Darmon construisent des points de Stark-Heegner et conjecturent qu'ils sont définis sur $H_D$. On peut penser qu'ils engendrent un large sous-groupe comme dans le Théorème~\ref{th:main}. Pour mesurer la difficulté d'une telle question, rappelons que lorsque $D\rightarrow +\infty$, on ne connaît aucune borne inférieure pour le nombre de classes $h(D)$.

\subsection{Notations.}\label{sec:intro:notation} Si $\pi$ est une représentation automorphe parabolique de $GL_n/\BmQ$, on note $L(s,\pi)$ sa fonction $L$, normalisée de sorte que l'équation fonctionnelle lie $s$ et $1-s$.

\subsection{Organisation des chapitres.}\label{sec:intro:articulation}
On a choisi de concentrer les arguments nouveaux dans les chapitres très courts~\ref{sec:demo} et~\ref{sec:general}. On espère ainsi rendre la lecture de la démonstration totalement transparente. Les chapitres sont articulés de la manière suivante. 

Dans le chapitre~\ref{sec:GZ} on rappelle le formalisme des courbes de Shimura, des points de Heegner et l'énoncé de la formule de Gross et Zagier. On y démontre le Lemme~\ref{lem:singularitepointe} qui sera important pour la suite. Dans le chapitre~\ref{sec:demo} on démontre le Théorème~\ref{th:main} à partir de la Proposition~\ref{prop:height}. Dans le chapitre~\ref{sec:general} on démontre la Proposition~\ref{prop:height}. Dans l'appendice~\ref{sec:Heegner} on démontre la Proposition~\ref{prop:height} sous la condition de Heegner usant d'arguments rudimentaires faisant appel aux pointes des courbes modulaires.


\section{Rappels: la formule de Gross et Zagier.}\label{sec:GZ}
Les résultats de ce chapitre sont connus. Nous n'avons pas trouvé une référence qui donne clairement les énoncés nécessaires à la démonstration du Théorème~\ref{th:main} dans le cas optimal de la condition~\eqref{eq:th:main:sgn}. C'est pourquoi il nous paraît judicieux de rappeler ici le \emph{strict nécessaire} des résultats qui mènent à la construction des points de Heegner sur les courbes elliptiques (une référence plus complète est~\cite{Zhan01:annals} mais il est demandé  que $N$ et $D$ soient premiers entre eux). Nous ne donnerons aucune démonstration mais renvoyons à  suffisamment de références bibliographiques pour que le lecteur consciencieux puisse reconstituer sans trop de difficulté des démonstrations complètes.

\subsection{Courbes modulaires et courbes de Shimura.}\label{sec:GZ:courbes}
Pour construire une courbe de Shimura, on a besoin de certaines \Lquote{données} dans la terminologie introduite dans~\cite{cong:AMS79:deli}. Dans cet article, on pourra se contenter de la définition rudimentaire suivante:
\begin{definition}[Données]\label{def:donnee}
On désigne par $\CmN=(N_1,N_2,(K_p)_{p|N_1N_2})$ la donnée de deux entiers $N_1\ge 1$ et $N_2\ge 1$ premiers entre eux et pour tout premier $p$ divisant $N_1N_2$ d'une extension quadratique $K_p$ de $\BmQ_p$, à isomorphisme près. On dira que $\CmN$ est de niveau $N_1N_2$. On impose de plus les conditions suivantes:
\begin{enumerate}
 \item[(i)] lorsque $p|N_2$, $K_p$ est un corps;
 \item[(ii)] le nombre de facteurs premiers de $N_2$ est pair;
 \item[(iii)] lorsque $p|N_2$ et $K_p=\BmQ_{p^2}$ est l'extension non ramifiée de $\BmQ_p$, $v_p(N_2)$ est impair;
 \item[(iv)] lorsque $p|N_1$ et $K_p=\BmQ_{p^2}$ est l'extension non ramifiée de $\BmQ_p$, $v_p(N_1)$ est pair.
\end{enumerate}
\end{definition}

\begin{remark} Observons qu'il y a au plus $15\times 7^{\omega(N)}$ données de niveau $N\ge 1$, où $\omega(N)$ est le nombre de diviseurs premiers de $N$.
\end{remark}

\begin{lemma}[Algèbre de quaternions] Soit $\CmN$ une donnée de niveau $N$.
À isomorphisme près, il existe une unique algèbre de quaternions $B$ sur $\BmQ$ qui est ramifiée en les premiers divisant $N_2$. Pour tout premier $p|N_1N_2$, il existe un plongement $K_p \hookrightarrow B\otimes_\BmQ \BmQ_p$ (morphisme de $\BmQ_p$-algèbres) qui est unique à $B_p^\times$-conjugaison près.
\end{lemma}

Ce lemme est bien connu et fait appel aux conditions (i) et (ii) de la Définition~\ref{def:donnee}. Dans la suite on pose $B_p=B\otimes_\BmQ \BmQ_p$ et on fixe un plongement $K_p
\hookrightarrow B_p$. On peut trouver une démonstration claire du lemme suivant dans \cite{Gros88}*{Proposition 3.4} (voir également les sections 2 et 3 de~\cite{HPS89} pour une description explicite lorsque $B_p$ est non déployée).
\begin{lemma}[Ordre de Bass]\label{lem:bass} Pour tout premier $p|N_1N_2$, il existe un ordre local $\CmR_p\subset B_p$ qui contient l'anneau des entiers de $K_p$ et qui est de discriminant réduit $p^{v_p(N)}$. Il est unique à $K^\times_p$-conjugaison près. Il existe un ordre global $\CmR \subset B$ tel que $\CmR\otimes_\BmZ \BmZ_p = \CmR_p$ pour tout $p|N_1N_2$ et $\CmR\otimes_\BmZ \BmZ_p \simeq M_2(\BmZ_p)$ pour tout $p\nmid N_1N_2$.
\end{lemma}

Soit $\BmA$ l'anneau des adèles sur $\BmQ$, $\BmA_f$ le sous-anneau des adèles finies et $\widehat{\BmZ}:=\varprojlim \BmZ/m\BmZ$. On pose $\widehat{B}:=B\otimes_\BmQ \BmA_f$ et $\widehat{\CmR}:=\CmR\otimes_\BmZ \widehat{\BmZ}$. L'ensemble $\widehat{\CmR}^\times$ est un sous-groupe ouvert compact de $\widehat{B}^\times$. Soit $G$ le groupe algébrique sur $\BmQ$ associé à $B^\times$, c'est-à-dire que $G(A)=(B\otimes_\BmQ A)^\times$ pour toute $\BmQ$-algèbre $A$. En particulier $G(\BmA_f)=\widehat{B}^\times$. On fixe un isomorphisme de $\BmR$-algèbres $B\otimes \BmR\simeq M_2(\BmR)$, de sorte que $G(\BmR)\simeq GL_2(\BmR)$. Soit $\BmS:=\MRes_{\BmC/\BmR}(\BmG_m)$ et $h:\BmS \rightarrow G(\BmR)$ le morphisme qui envoie $x+iy$ sur $\Mdede{x}{y}{-y}{x}$.
\begin{definition}
On note $X_\CmN$ la courbe de Shimura associée à la donnée de Shimura $(G,h,\widehat{\CmR}^\times)$. C'est une courbe projective lisse connexe sur $\BmQ$.
\end{definition}

\begin{example} Lorsque $N_2=1$ et pour tout $p|N_1$, l'extension $K_p$ est déployée (c'est-à-dire isomorphe à $\BmQ_p\oplus \BmQ_p$), $X_\CmN$ est la courbe modulaire $X_0(N_1)$.
\end{example}

L'ordre $\CmR$ n'est pas unique à $B^\times$-conjugaison près, mais il est unique à $\widehat{B}^\times$-conjugaison près. Ainsi $X_\CmN$ ne dépend pas du choix de $\CmR$ à isomorphisme près. Dans la terminologie de~\cite{Eich55}, la donnée $\CmN$ détermine le \Lquote{type} de $\CmR$. Deux bonnes références pour des propriétés plus fines de $X_\CmN$ sont \cites{Zhan01:annals,book:Nekovar:euler}.\footnote{Dans~\cite{Zhan01:annals}, $K_p$ est non ramifiée pour tout $p|N_1N_2$.}

Pour alléger les notations, on note $X=X_\CmN$ dans la suite de ce paragraphe et dans les deux suivants. L'uniformisation complexe est donnée par:
\begin{equation}\label{XC}
 X(\BmC)= B^\times \SB (\BmC-\BmR) \times \widehat{B}^\times /\widehat{\CmR}^\times
\end{equation}
(du moins lorsque $N_2>1$). Par le théorème d'approximation forte~\cite{book:PR}, on a $B_+^\times\widehat{\CmR}^\times=\widehat{B}^\times$ où $B_+^\times$ est le sous-groupe des éléments de $B^\times$ de norme réduite positive. Ainsi $X$ est géométriquement connexe et $X(\BmC)=\Gamma \SB \FmH$ où $\FmH$ est le demi-plan de Poincaré et $\Gamma \subset SL_2(\BmR) \subset GL_2(\BmR)$ est le réseau arithmétique cocompact défini par $\Gamma:=B_+^\times \cap \widehat{\CmR}^\times$. La courbe $X$ n'est rationnelle (isomorphe à $\BmP_\BmQ^1$) que dans un nombre fini de cas qui seront automatiquement exclus dans la suite.

\begin{example}
Lorsque $N_2=1$, il faut modifier légèrement l'uniformisation complexe~\eqref{XC}. Le réseau $\Gamma$ n'est pas cocompact et il manque au quotient $\Gamma\SB \FmH$ un nombre fini de pointes pour former la surface de Riemann compacte $X(\BmC)$. L'exemple de la courbe modulaire $X_0(N_1)$ est bien connu.
\end{example}

\remark Dans les articles~\cites{Buzz97,BD96,BC91,BD07} il est demandé que $N_2$ soit sans facteur carré. Dans ce cas les ordres locaux $\CmR_p$ et l'ordre $\CmR$ sont des ordres d'Eichler. Un ordre d'Eichler est l'intersection de deux ordres maximaux.

\subsection{Construction d'un diviseur rationnel de degré $1$.}\label{sec:GZ:xi}
À partir de maintenant on suppose que $X$ n'est pas rationnelle. Soit $J$ la Jacobienne de $X$. C'est une variété abélienne sur $\BmQ$ de dimension le genre de $X$.

Dans \cites{Zhan01:annals,Zhan01b}, Zhang construit un diviseur rationnel sur $X$ (qu'il baptise \Lquote{diviseur de Hodge} pour l'analogie des géométries complexes et d'Arakelov). Ce diviseur joue un rôle important dans l'énoncé de la formule de Gross-Zagier générale et dans le fait que l'image des points de Heegner sur les courbes elliptiques restent définis sur $H_D$ (et pas sur une extension de $H_D$). 
\begin{lemma}[Zhang] Il existe un unique diviseur $\xi\in \MPic(X)\otimes_\BmZ \BmQ$ avec les propriétés suivantes:
\begin{enumerate}
 \item[(i)] $\xi$ est de degré $1$,
 \item[(ii)] pour tout entier $n$ premier à $N_1N_2$, $T_n\xi=\deg(T_n)\xi$, où $T_n$ est le $n$-ième opérateur de Hecke.
\end{enumerate}
On note $\iota:X\rightarrow J$ le morphisme défini sur $\BmQ$ induit par ce diviseur. 
\end{lemma}
Rappelons la construction de $\xi$. Lorsque $N_1N_2$ est suffisamment grand (typiquement lorsque le groupe $\Gamma$ est sans torsion) on peut prendre pour $\xi$ la classe du diviseur canonique $\frac{1}{g}[\Omega^1_X]$. Dans le cas général, il suffit de considérer la projection $X'\rightarrow X$ d'une courbe de Shimura $X'$ de niveau suffisamment grand et de projeter le diviseur $\xi'$ de $X'$ sur $X$. Le lemme est alors une conséquence facile de la formule d'Hurwitz, voir~\cite{CV05}*{section 3.5} pour plus de détails.

\begin{example}
Lorsque $N_2=1$, on peut choisir pour $\xi$ un diviseur de degré $1$ supporté en les pointes, typiquement $(i\infty)$ comme dans \cite{GZ}. L'unicité de $\xi$ est compatible avec le théorème de Manin-Drinfeld~\cites{Mani72,Drin73}.
\end{example}

\subsection{Points de Heegner.}\label{sec:GZ:heegner}
Soit $K=\BmQ(\sqrt{D})$ un corps quadratique imaginaire tel que $K\otimes_\BmQ \BmQ_p$ soit isomorphe à $K_p$ pour tout $p|N_1N_2$. Il existe un plongement $K\hookrightarrow B$ (unique à $B^\times$-conjugaison près). On en fixe un et on considère dans la suite que $K$ est inclus dans $B$. On choisit l'ordre $\CmR$ du Lemme~\ref{lem:bass} de telle sorte que $\CmO_K\subset \CmR$.

Soit $z$ l'unique point de $\FmH\subset \BmC-\BmR$ qui est fixe par l'action de $K^\times$ (cette action est induite par l'inclusion $K^\times\subset B^\times\simeq GL_2(\BmR)$).
Soit $z_D\in X(\BmC)$ le point qui est défini par la double classe $B^\times[z,1]\widehat{\CmR}^\times$ dans l'uniformisation~\eqref{XC}. 

D'après la théorie de la multiplication complexe, voir~\cites{conf:complex-mult} ou \cite{book:Nekovar:euler}*{section 2.4}, le point $z_D$ est algébrique, défini sur $K^{\Tab}$ et l'action par le groupe de Galois $\MGal(K^{\Tab}/K)$ est donnée de la manière suivante. Soit $\Mrec: K^\times \SB \widehat{K}^\times \rightarrow \MGal(K^{\Tab}/K)$ l'application de réciprocité du corps de classes. Alors pour tout $t\in \widehat{K}^\times$ le conjugué de la double classe $B^\times[z,1]\widehat{\CmR}^\times$ par l'automorphisme $\Mrec(t)$ est la double classe $B^\times[z,t]\widehat{\CmR}^\times$.

Les doubles classes $B^\times[z,1]\widehat{\CmR}^\times$ et $B^\times[z,t]\widehat{\CmR}^\times$ sont égales si et seulement si il existe $b\in B^\times$ et $k\in \widehat{\CmR}^\times$ tels que $b\cdot z=z$ et $b=tk$. La première condition est équivalente à $b\in K^\times$ et alors la deuxième condition implique $k\in \widehat{K}^\times \cap \widehat{\CmR}^\times=\widehat{\CmO_K}^\times$. Le stabilisateur de $z_D$ par Galois est donc $\Mrec(K^\times \SB K^\times / \widehat{\CmO_K}^\times)$, qui n'est autre que $\MGal(H_D/K)$. En particulier, le corps de définition de $z_D$ est exactement $H_D$.

\subsection{Théorème de Wiles.}\label{sec:GZ:wiles}
Soit $E$ une courbe elliptique de conducteur $N$. On sait grâce aux travaux de Wiles~\cite{Wile95} et Taylor-Wiles~\cite{TW95} (le cas général est établi dans~\cite{BCDT01}), que l'on peut associer à $E$ une forme modulaire $f$ primitive de poids $2$ et de niveau $N$ et qu'il existe un morphisme non constant $\varphi: X_0(N)\rightarrow E$ défini sur $\BmQ$ (rappelons que $f$, vue comme forme différentielle sur $X_0(N)$ est proportionnelle au pullback par $\varphi$ d'une différentielle de Néron sur $E$). On demande traditionnellement que l'image de la pointe $i\infty$ soit l'origine de $E$.

Le morphisme $\varphi$ est \Lquote{une paramétrisation de Weil} au sens de~\cite{MS74}. On peut trouver des informations supplémentaires dans le rapport~\cite{bour:mazu} (comme les définitions d'une \Lquote{courbe de Weil forte} et d'une \Lquote{paramétrisation de Weil forte}). Rappelons ici (\cite{MS74}*{Lemme 1}), que $\varphi$ est étale en la pointe $i\infty$. En effet le premier coefficient de Fourier de $f$ est non nul par la théorie des formes nouvelles.

\subsection{Correspondance de Jacquet-Langlands.}\label{sec:GZ:JL}
On peut associer à $f$ une unique représentation automorphe cuspidale $\pi$ de $GL_2(\BmA)$, voir~\cites{book:gelb:adel,cong:Lfunc04:cogd}. On note $\pi\simeq \otimes_v \pi_v$ sa décomposition en produit tensoriel de représentations locales. Une conséquence de la correspondance de Jacquet-Langlands~\cite{book:JL}*{part III} est qu'il existe également un morphisme non constant $\varphi:X_\CmN\rightarrow E$ quelque soit la donnée $\CmN=(N_1,N_2,(K_p)_{p|N_1N_2})$ qui vérifie la condition suivante (le groupe $G/\BmQ$ et l'ordre $\CmR$ sont définis au \S~\ref{sec:GZ:courbes}):
\begin{hypothese}\label{def:JL} Pour tout $p|N_2$, la représentation $\pi_p$ est de carré intégrable. La représentation automorphe cuspidale $\pi'$ de $G(\BmA)$ associée à $\pi$ par la correspondance de Jacquet-Langlands admet un vecteur invariant par $\widehat{\CmR}^\times$.
\end{hypothese}
\begin{remark} Cette condition est \Lquote{locale}. En effet une formulation équivalente est la suivante. Pour tout $p|N_1$ (resp. $p|N_2$), la représentation $\pi_p$ (resp. $\pi'_p$) admet un vecteur invariant par le sous-groupe ouvert compact $\CmR^\times_p\subset G(\BmQ_p)$.
\end{remark}

L'objet de ce paragraphe est de rappeler brièvement les étapes de la construction du morphisme $\varphi$. On note $V_{\pi'}\subset L^2(Z(\BmA^\times)G(\BmQ)\SB G(\BmA))$ l'espace de la représentation $\pi'$ (où $Z\simeq \BmG_m$ est le centre de $G$). La Condition~\ref{def:JL} dit que le $\BmC$-espace vectoriel $V^{\widehat{\CmR}^\times}_{\pi'}$ des éléments qui sont invariants par $\widehat{\CmR}^\times$ est non-réduit à $\{0\}$. La $\BmQ$-algèbre de Hecke $\CmH:=\BmQ[\widehat{\CmR}^\times \SB \widehat{B}^\times / \widehat{\CmR}^\times]$ agit naturellement sur cet espace et la représentation est isotypique. On note $\CmI$ le noyau de cette action.

L'algèbre $\CmH$ agit par isogénies sur  la Jacobienne $J_\CmN$ de $X_\CmN$. Le quotient de $J_\CmN$ par $\CmI$ est une variété abélienne $E'$, voir par exemple~\cite{book:Nekovar:euler}*{(1.19)}. Elle est de dimension $1$ (courbe elliptique).
 
Par la relation d'Eichler-Shimura, la cohomologie $l$-adique de la courbe elliptique $E'$ est isomorphe à celle de $E$ (en tant que $\MGal(\overline{\BmQ}/\BmQ)$-module), pour tout $l\nmid N_1N_2$. Le théorème de Faltings~\cites{Falt83} montre que $E'$ est isogène à $E$.

La composée des trois applications $X_\CmN \stackrel{\iota}{\rightarrow} J_\CmN \rightarrow E' \sim E$ est le morphisme non constant $\varphi$ que l'on voulait construire.

\subsection{Mesures et points images.}\label{sec:GZ:image}
Grâce à l'application $\varphi$, on peut \Lquote{pousser} des objets de $X_\CmN$ sur $E$. Par exemple l'image $\varphi(z_D)$ du point de Heegner du \S~\ref{sec:GZ:heegner} appartient à $E(H_D)$ et jouera un rôle central dans la suite. Remarquons que le corps de définition de $\varphi(z_D)$ est en général strictement inclus dans $H_D$ (mais l'indice est borné par le degré de $\varphi$).

On peut aussi considérer des mesures images. Dans la suite on aura besoin du lemme suivant (qui apparaît indépendamment dans~\cite{BP07}*{Lemma 3.6}):
\begin{lemma}\label{lem:singularitepointe} Soit $\mu$ la mesure hyperbolique sur $X_\CmN(\BmC)\simeq \Gamma\SB \FmH$ qui provient de l'uniformisation par le demi-plan de Poincaré. La mesure image $\varphi_*\mu$ n'est pas une mesure de Haar de $E(\BmC)$.
\end{lemma}
\begin{proof}
Soit $g:\BmC\rightarrow E(\BmC)$ le revêtement universel de $E(\BmC)$. La mesure de Haar $\nu$ sur $E(\BmC)$ est la mesure quotient d'une mesure de Haar sur $\BmC$.

Si $N_2=1$, considérons l'uniformisation locale de la pointe $i\infty$ (en tant que surface de Riemann). Elle est induite par l'application $\eta:\Gamma\SB \FmH\rightarrow \BmD$ donnée par $z\mapsto e^{2i\pi z/h}$, où $h>0$ la largeur de la pointe $i\infty$, c'est-à-dire que le stabilisateur de $i\infty$, est $\Gamma_\infty=\Mdede{1}{h\BmZ}{0}{1}$. Considérons les morphismes:
\begin{equation}
 \BmD 
 \stackrel{\eta}{\longleftarrow}
 X_\CmN(\BmC) 
 \stackrel{\varphi}{\longrightarrow}
  E(\BmC) 
  \stackrel{g}{\longleftarrow}
  \BmC
\end{equation}
Soit $A=\{x+iy,\ 0\le x < h, Y < y < \infty\}\subset \FmH$. Lorsque $Y>0$ est assez grand on peut considérer $A$ comme un sous-ensemble de $X_\CmN(\BmC)$. On va montrer que les volumes de $\varphi(A)$ pour les mesures respectives $\varphi_*\mu$ et $\nu$ sont distincts. Pour le premier volume on écrit:
\begin{equation}
 \varphi_*\mu(\varphi(A)) = \mu(\varphi^{-1}(\varphi(A))) \ge \mu(A) \sim Y. 
\end{equation}
Pour le second, soit $B=\eta(A)\subset \BmD$ qui est la boule ouverte de centre $0$ et de rayon $r:=e^{-2\pi Y/h}$. Lorsque $Y$ est assez grand, la restriction $\eta|_A:A\rightarrow B$ est inversible et on a: 
 \begin{equation}
  \varphi(A)=\varphi\circ \eta|^{-1}_A (B).
 \end{equation}
L'application $\varphi\circ \eta|^{-1}_A$ est holomorphe, donc $\varphi(A)$ est inclus dans une boule euclidienne de rayon proportionnel à $r$ (en fait le diamètre de $\varphi(A)$ est vraiment proportionnel à $r$ parce $\varphi$ est étale en $i\infty$). On en déduit que $\nu(\varphi(A))\ll r^2=e^{-4\pi Y/h}$. Il est clair que les deux volumes $\varphi_*\mu(\varphi(A))$ et $\nu(\varphi(A))$ sont distincts lorsque $Y$ est suffisamment grand.

Supposons que $N_2>1$. Alors $X_\CmN(\BmC)=\Gamma\SB \FmH$. Considérons cette fois la projection $\rho:\FmH\rightarrow \Gamma\SB \FmH$ et les morphismes:
\begin{equation}
 \FmH 
  \stackrel{\rho}{\longrightarrow}
 \Gamma\SB\FmH 
   \stackrel{\varphi}{\longrightarrow}
  E(\BmC)
     \stackrel{g}{\longleftarrow}
   \BmC.
\end{equation}
L'application composée $\varphi\circ \rho:\FmH \rightarrow \Gamma\SB \FmH \rightarrow E(\BmC)$ est nécessairement ramifiée (puisque le revêtement universel du tore $E(\BmC)$ est $g:\BmC\rightarrow E(\BmC)$ et que $\BmC$ et $\FmH$ ne sont pas isomorphes). Soit $x\in \FmH$ un point d'indice de ramification $e>1$. Soit $A$ une boule de centre $x$ et de rayon $r$ plus petit que le rayon d'injectivité de $\rho$ (en $x$). On va montrer que les volumes de $\varphi\circ \rho(A)$ pour les mesures respectives $\varphi_*\mu$ et $\nu$ sont distincts lorsque $r$ est suffisamment petit. Pour le premier volume, on a la minoration:
\begin{equation}
 \varphi_*\mu(\varphi\circ \rho(A))= \mu(\varphi^{-1}\varphi\circ \rho(A)) \ge \mu(\rho(A)) = \mu(A) \sim r^2.
\end{equation}
Pour le second, on observe que $\varphi\circ \rho(A)$ est incluse dans une boule euclidienne de rayon $\sim r^e$, de sorte que $\nu(\varphi\circ \rho(A))\ll r^{2e}$. Comme $e>1$, la conclusion est claire.
\end{proof}

\begin{remark} Il est possible de faire le même type de raisonnement à l'aide d'uniformisations $p$-adiques de $X_\CmN$.
\end{remark}

\begin{remark} La démonstration du Lemme~\ref{lem:singularitepointe} utilise uniquement le fait que $\varphi_\BmC$ est un morphisme non constant $X_\CmN(\BmC)\rightarrow E(\BmC)$. Dans \cite{Mazu91}, B.~Mazur montre que l'existence d'un tel morphisme est en fait équivalent au théorème de Wiles (où l'on demande en plus que $\phi$ soit définie sur $\BmQ$).
\end{remark}

\subsection{Fonctions $L$ de Rankin-Selberg.}\label{sec:GZ:RS}
Désignons par $\chi\in\Mclgr_D$ un caractère du groupe des classes d'idéaux de $K$. On note $L(s,f\times \chi)$ la convolution de Rankin-Selberg de $f$ avec l'induite quadratique de $\chi$, voir \cites{GZ,Zhan01b,book:JLII} pour plus de détails. On peut montrer que $L(s,f\times \chi)$ est auto-duale. Avec l'hypothèse~\eqref{S}, le signe de l'équation fonctionnelle ne dépend pas du caractère $\chi$. Il ne dépend donc que de $E$ et $D$, et on le note $\Msgn(E,D)\in \{\pm 1\}$.

\begin{remark} Le signe $\Msgn(E,D)$ est produit de facteurs $\epsilon$ locaux. Il n'est pas nécessaire de spécifier un caractère additif local car le facteur $\epsilon$ n'en dépend pas, voir~\cite{Gros88}*{Sections 5 et 6}. On note donc $\epsilon_v(1/2,f\times \chi)\in \{\pm 1\}$ le facteur epsilon local (normalisation autoduale de Tate). On peut vérifier que $\epsilon_\infty(1/2,f\times \chi)=-1$ et que pour tout premier $p$ qui ne divise pas $ND$, on a $\epsilon_p(1/2,f\times \chi)=1$. Le calcul de $\Msgn(E,D)$ se résume donc à un nombre fini de considérations locales.
\end{remark}

\begin{example}\label{ex:localroot} Pour les places restantes (c'est-à-dire \emph{lorsque} $p|ND$), il est possible de déterminer $\epsilon_p(1/2,f\times \chi)$ explicitement lorsque la ramification commune de $f$ et $K$ n'est pas trop importante. D'après \cite{Gros88}*{Proposition 6.3} (voir également \cite{Zhan01b}*{section 3.1} ou \cite{book:Nekovar:asterisque}*{équation (12.6.2.4)}), on a en effet:
\begin{equation*}
 \epsilon_p(1/2,f\times \chi)\chi_{D,p}(-1)=
 \begin{cases}
  \chi_{D,p}(N)&\text{si $(p,D)=1$,}\\
  \chi_{D,p}(N)&\text{si $p|D$ et $(p,N)=1$,}\\
  -a_p(f)\chi_{D,p}(N)&\text{si $p|D$ et $p||N$.}
 \end{cases}
\end{equation*}
Lorsque les facteurs carrés de $N$ sont premiers à $D$, on a donc
\begin{equation}
 \Msgn(E,D)=-\chi_D(N)\prod_{p|(N,D)} -a_p(f).
\end{equation}
Rappelons~\cite{Tate74} que pour $p>2$ avec $p|N$, on a $a_p(f)=+1$ (resp. $a_p(f)=-1$) si et seulement si $E$ a réduction multiplicative scindée (resp. non scindée).
\end{example}

\subsection{Conséquences d'une \'equation fonctionnelle impaire.}\label{sec:GZ:equation}
Dans toute la suite on fera l'hypothèse fondamentale suivante (qui est aussi~\eqref{eq:th:main:sgn} dans le Théorème~\ref{th:main}):
\begin{equation}\label{signe-1}
 \Msgn(E,D)=-1.
\end{equation}
Observons que cette condition n'est pas très contraignante. Par exemple elle est toujours satisfaite lorsque $(D,N)=1$ et $\chi_D(N)=1$.

On définit la donnée $\CmN=(N_1,N_2,(K_p)_{p|N_1N_2})$ de la manière suivante. L'entier $N_1$ est le produit des $p^{v_p(N)}$ où $p$ parcourt les premiers $p$ tels que $\epsilon_p(1/2,f\times \chi)=-\chi_{D,p}(-1)$. L'entier $N_2$ est divisible par $N/N_1$, tous ses facteurs premiers divisent $N/N_1$, et on le choisit suffisamment grand pour la Condition~\ref{def:JL} soit satisfaite. Il n'est pas difficile\footnote{En fait, sauf pour quelques cas très ramifiés, on peut choisir $N_2:=N/N_1$ d'après la proposition ~6.3 de~\cite{Gros88}, voir aussi~\cite{GP91}.} de vérifier que $N_2=O_E(1)$. Cela vient du fait qu'il n'y a qu'un nombre fini d'extensions quadratiques de $\BmQ_p$ pour $p|N$. Pour tout $p|N_1N_2$, on pose $K_p\simeq K\otimes_\BmQ \BmQ_p$. 

Le fait que la donnée $\CmN$ vérifie la hypothèse (ii) de la Définition~\ref{def:donnee} est équivalent à~\eqref{signe-1}. Le fait que la donnée $\CmN$ vérifie les hypothèses (i), (iii) et (iv) est conséquence de la formule de Tunnell~\cite{Tunn83}, voir aussi~\cite{Wald85b} et \cite{Gros88}*{Section 5} pour plus de détails ainsi que la démonstration élégante~\cite{Pras07}. On peut donc appliquer les constructions des \S~\ref{sec:GZ:courbes} à \ref{sec:GZ:heegner} (qui fournissent une courbe de Shimura $X_\CmN$ et un point de Heegner $z_D\in X_\CmN(H_D)$).

\begin{example}
Rappelons que la condition de Heegner est le fait que tous les facteurs premiers de $N$ sont décomposés par $K$. La condition de Heegner implique~\eqref{signe-1} (voir les formules données dans l'exemple~\ref{ex:localroot}). Dans ce cas, la donnée $\CmN$ est $(N,1,(\BmQ_p\oplus \BmQ_p)_{p|N})$ et $X_\CmN$ est la courbe modulaire $X_0(N)$. Les points de Heegner sont alors comme définis dans~\cite{GZ}*{chapter I}. Il faut voir~\eqref{signe-1} comme une version optimale de la condition de Heegner. 
\end{example}

En fait la formule de Tunnell contient plus d'informations. La donnée $\CmN$ a été choisie de telle sorte que pour tout $p|N_2$, la représentation $\pi_p$ de $GL_2(\BmQ_p)$ est de carré intégrable, et pour tout premier $p$ la représentation $\pi'_p$ de $B^\times_p$ possède une forme linéaire invariante par le sous-groupe $K^\times_p$.

Rappelons que $\CmN$ et plus particulièrement $N_2$ sont choisis de sorte que la seconde hypothèse de la Condition~\ref{def:JL} soit satisfaite (vecteur invariant par $\widehat{\CmR}^\times$). On peut donc appliquer les constructions des \S~\ref{sec:GZ:wiles} à \ref{sec:GZ:image} (existence d'un morphisme non constant $\varphi:X_\CmN\rightarrow E$).

\subsection{Énoncé.}\label{sec:GZ:formule}
La formule de Gross et Zagier a été établie initialement dans~\cite{GZ} sous la condition de Heegner\footnote{le dernier Chapitre V de \cite{GZ} explique que la démonstration devrait s'adapter au cas $(N,D)=1$.}, puis dans~\cite{Zhan01b} lorsque $(N,D)$ est sans facteur carré et par X.~Yuan, Zhang et W.~Zhang~\cite{YZZ08} dans le cas général.

La composition de $\chi$ avec l'application de réciprocité du corps de classes (notée $\Mrec$ dans le \S~\ref{sec:GZ:heegner}) est un caractère du groupe de Galois $\MGal(H_D/K)$. On introduit
\begin{equation}
 z_\chi:=\frac{1}{h(D)}\sum_{\SmA\in \MCl_D}\chi(\SmA)\varphi(z^\SmA_D),
\end{equation}
qui est la $\chi$-composante de $\varphi(z_D)$. On note $\widehat{h}:E(\overline{\BmQ})\rightarrow \BmR_+$ la hauteur de Néron-Tate (voir par exemple \cite{book:Silv:elliptic}*{Chapitre VIII, Section 9}). Rappelons que $\widehat{h}$ est une forme quadratique qui est nulle pour les points de torsion et dont la $\BmC$-extension à $E(\overline{\BmQ})\otimes_\BmZ \BmC$, que l'on notera encore $\widehat{h}$, est définie positive.

\begin{theorem}[Formule de Gross et Zagier]
 Il existe une constante $\alpha>0$ qui ne dépend que de $\varphi:X_\CmN\rightarrow E$ telle que:
\begin{equation}\label{formuleGZ}
 L'(1/2,f\times \chi)
 = \alpha L(1,\chi_D)
 \widehat{h}(z_\chi).
\end{equation}
\end{theorem}

Comme la donnée $\CmN$ dépend de $D$, la constante $\alpha$ dépend implicitement de $D$. Comme $E$ est fixée, et qu'il n'existe qu'un nombre fini de données $\CmN$ de niveau $N_1N_2=O_E(1)$, il est clair que $\alpha\in]0,\infty[$ est bornée lorsque $D$ varie; en particulier $\alpha \gg_E 1$.

\section{Le sous-groupe engendré par les points de Heegner.}\label{sec:demo}
\subsection{Non-annulation et rang.}\label{subsec:nonannulation}\label{sec:demo:nonannulation}
Si le réel $L'(\Mdemi,f\times \chi)$ est non nul, l'espace $\chi$-isotypique $(E(H_D)\otimes_\BmZ \BmC)^\chi$ est non réduit à $\{0\}$. En effet la hauteur de $z_\chi$ est alors strictement positive par la formule de Gross et Zagier~\eqref{formuleGZ}. Ainsi le rang du sous-groupe engendré par le point $\varphi(z_D)$ et ses conjugués par Galois est au moins égal au nombre de caractères $\chi$ tels que $L'(\Mdemi,f\times \chi)$ est non nul.

Rappelons que le nombre total de caractères $\chi\in \Mclgr_D$ est égal au nombre de classes $h(D)$, et que le théorème de Siegel~\cite{Sieg:classnumber} affirme que $h(D)\gg_\epsilon |D|^{\Mdemi-\epsilon}$ pour tout $\epsilon>0$.

\subsection{Majoration de sous-convexité.}\label{sec:demo:subconvex}
D'après la majoration de sous-convexité établie par Ph.~Michel \cite{Mich04}*{Theorem 2}, on a (avec $\delta:=1/1057$):
\begin{equation}\label{subconvex}
 L'(\Mdemi,f\times \chi)\ll_f |D|^{\Mdemi-\delta}, 
 \qtext{pour tout $\chi\in \Mclgr_D$.}
\end{equation}
On peut noter également que dans le preprint~\cite{MV09}, Michel-Venkatesh donnent une démonstration de cette majoration de sous-convexité dans le cas le plus général.

\subsection{Évaluation du premier moment.}\label{sec:demo:moment} On considère le moment d'ordre un. En appliquant la formule de Gross et Zagier (une seconde fois!), on obtient:
\begin{equation}\label{GZ}
  \frac{1}{h(D)}
  \sum_{\chi\in \Mclgr_D} 
  L'(\Mdemi,f\times \chi)
   =
   \alpha L(1,\chi_D)\widehat{h}(\varphi(z_D)).
\end{equation}

\begin{proposition}\label{prop:height} Il existe une constante $c>0$ qui ne dépend que de $\varphi:X_\CmN\rightarrow E$ et telle que pour tout discriminant $\pabs{D}$ suffisamment grand on a $\widehat{h}(\varphi(z_D))\ge c$.
\end{proposition}
{\it Démonstration du Théorème~\ref{th:main}.}\
Grâce à cette proposition (qui sera démontrée dans le chapitre~\ref{sec:general}), on peut conclure la démonstration du Théorème~\ref{th:rang}. En effet le membre de droite de l'équation~\eqref{GZ} est donc minoré par $\gg_\epsilon |D|^{-\epsilon}$ pour tout $\epsilon>0$ (on utilise la borne de Siegel pour minorer $L(1,\chi_D)$). En appliquant la majoration~\eqref{subconvex}, on en déduit qu'il existe au moins $\gg_\epsilon \pabs{D}^{\delta-\epsilon}$ caractères $\chi$ tels que $L'(\Mdemi,f\times \chi)$ est non nulle. Le Théorème~\ref{th:rang} en découle par la discussion du paragraphe~\ref{subsec:nonannulation}.

\remark Un corollaire de la formule de Gross et Zagier~\eqref{formuleGZ} est que $L'(\Mdemi,f\times \chi)$ est toujours positif ou nul. Ce fait n'a pas été utilisé dans la démonstration.

\section{Minoration de la hauteur.}\label{sec:general}
Pour un point $x$ de $E(\overline{\BmQ})$ ou $X_\CmN(\overline{\BmQ})$, on désigne par $\CmO(x)$ l'ensemble de ses conjugués par Galois (que l'on peut voir comme un $1$-cycle).

\subsection{Le théorème de Duke.}\label{sec:general:Duke}
Le théorème de Duke dit que $\CmO(z_D)\subset X$ est uniformément distribué selon $\mu$ lorsque $D\rightarrow -\infty$. Ce théorème est établi dans~\cite{Duke88} lorsque $N=1$, dans \cite{DFI4} lorsque $\CmN=(N,1,(\BmQ_p\oplus \BmQ_p)_{p|N})$ (courbes modulaires) et dans~\cite{Zhan05} dans le cas général. On peut également consulter~\cite{Mich04}*{Section 6} et \cite{HM06} pour des résultats complémentaires.

On en déduit que $\varphi_* \CmO(z_D)$ est uniformément distribué dans $E(\BmC)$ selon la mesure image $\varphi_* \mu$ lorsque $D\rightarrow -\infty$: il suffit de vérifier le critère de Weyl. Soit $F:E(\BmC)\rightarrow \BmR$ continue avec $\int F d\varphi_*\mu=0$. Alors 
\begin{equation}
\frac{1}{h(D)}\sum_{x\in \varphi_* \CmO(z_D)}
F(x)
=\frac{1}{h(D)}
\sum_{x\in \CmO(z_D)} F\circ \varphi (x)
\end{equation}
tend vers $\int F\circ \varphi d\mu=0$ lorsque $D\rightarrow -\infty$.
 
Comme $\varphi:X_\CmN \rightarrow E$ est définie sur $\BmQ$, il est clair que le cycle $\varphi_* \CmO(z_D)$ est un multiple du cycle $\CmO(\varphi(z_D))$. On en déduit que $\CmO(\varphi(z_D))$ est uniformément distribué selon $\phi_* \mu$ lorsque $D\rightarrow -\infty$.

\subsection{Le théorème de Szpiro-Ullmo-Zhang.}\label{sec:general:USZ}
Le théorème de Szpiro, Ullmo et Zhang \cite{SUZ97}*{Theorem 1.2} dans le cas des courbes elliptiques est l'énoncé suivant. Soit $(x_n)$ une suite de points de $E(\overline{\BmQ})$ deux à deux distincts et tels que $\widehat{h}(x_n)\rightarrow 0$ quand $n\rightarrow +\infty$. Alors $\CmO(x_n)$ est uniformément distribué selon la mesure de Haar $\nu$ sur $E(\BmC)$ lorsque $n\rightarrow +\infty$.

\subsection{Démonstration de la Proposition~\ref{prop:height}.}
D'après le Lemme~\ref{lem:singularitepointe}, les mesures $\phi_* \mu$ et $\nu$ sont distinctes. Par contraposée, on déduit des théorèmes de Duke et Szpiro-Ullmo-Zhang qu'il existe un réel $c>0$ tel que $\widehat{h}(\varphi(z_D))\ge c$ lorsque $\pabs{D}$ est suffisamment grand. Cela conclut la démonstration de la Proposition~\ref{prop:height} et du Théorème~\ref{th:rang}.

\section{Appendice: Minoration sous la condition de Heegner.}\label{sec:Heegner}
Dans ce chapitre on démontre la Proposition~\ref{prop:height} lorsque la condition de Heegner est satisfaite, c'est-à-dire que $X_\CmN=X_0(N)$ est la courbe modulaire. En fait on démontrera une minoration plus fine~\eqref{minoration:heegner}. L'idée de la démonstration provient de discussions avec A.~Venkatesh. Notons que la minoration~\eqref{minoration:heegner} est optimale à une constante multiplicative près d'après nos travaux~\cite{Temp:shifted}, voir la Remarque~\ref{rem:optimal}.

\subsection{Lemmes quantitatifs.}\label{sec:Heegner:burgess}
Définissons la quantité suivante:
\begin{equation}
 \CmL_D:=\Mdemi \log \pabs{D} +\frac{L'}{L}(1,\chi_D).
\end{equation}
Il est clair que l'on a $\CmL_D \ll_\epsilon \pabs{D}^\epsilon$ pour tout $\epsilon>0$. On a aussi $\CmL_D\rightarrow +\infty$ lorsque $D\rightarrow -\infty$. C'est un ingrédient important pour conclure à la minoration~\eqref{minoration:heegner}, où le $O_E(1)$ peut être négatif. Il découle de la majoration de sous-convexité de Burgess ou bien de la loi de Weyl uniforme pour les zéros de $L(s,\chi_D)$. Précisément d'après \cite{Temp:Eisenstein}*{Proposition 3.2}:
\begin{lemma}\label{lem:LD:minoration}
On a $\CmL_D > \frac{1}{3} \log |D|$ pour $\pabs{D}$ suffisamment grand.
\end{lemma}

Le lemme suivant est classique et découle de la formule de période de Hecke et de la formule limite de Kronecker, voir~\cite{Temp:Eisenstein}*{Section 2}:
\begin{lemma}\label{lem:heegnerpointe} Soit $\eta(z)=q^{1/24}\prod\limits^\infty_{n=1}(1-q^n)$ la fonction $\eta$ de Dedekind. 
\begin{equation}
 \frac{-1}{h(D)}
 \sum_{\SmA\in\MCl_D}
 \log 
 \pabs{\Mim z^\SmA_D \cdot \eta(z^\SmA_D)^4}
 =
 \CmL_D+\log 2-\gamma.
\end{equation}
En particulier on a
\begin{equation}
 \frac{1}{h(D)}
 \sum_{\SmA\in\MCl_D}
 \Mht z^\SmA_D
 = \frac{3}{\pi}\CmL_D +O(1)
\end{equation}
où, pour un point $z\in \FmH$ on pose $\Mht z:=\max\limits_{\gamma \in SL_2(\BmZ)} \Mim \gamma z$.
\end{lemma}
En fait on aura besoin de la version en niveau $N$ suivante (la démonstration est laissée au lecteur):
\begin{equation}\label{heegnerpointe}
  \frac{1}{h(D)}
 \sum_{\SmA\in\MCl_D}
 \Mht_N z^\SmA_D
 = \Mvol(X_0(N))^{-1}\CmL_D +O_N(1),
\end{equation}
où $\Mht_N z:=\max\limits_{\gamma \in \Gamma_0(N)} \Mim \gamma z$.

\subsection{Hauteurs.}\label{subsec:naive}\label{sec:Heegner:naive} Soit $\varrho:E_{/\BmQ}\hookrightarrow \BmP^2_\BmQ$ un plongement donné par une équation de Weierstrass (en particulier l'image de l'élément neutre par $\varrho$ est le point de coordonnées $[0:1:0]$). D'après le formalisme des hauteurs, on sait que
\begin{equation}\label{NT-naive}
  \widehat{h}(z)=h(\varrho(z))+O_E(1),
  \qtext{pour tout $z\in E(\overline{\BmQ})$.}
\end{equation}
Ici $h$ désigne la hauteur na\"ive sur $\BmP^2_{\BmQ}$, c'est-à-dire que pour un point $P=[x_1:x_2:x_3]\in \BmP^2(K)$ défini sur le corps de nombres $K$ on a:
\begin{equation}
  h(P):=\frac{1}{[K:\BmQ]}\sum_v \log \max(\pnorm{x_1}_v,\pnorm{x_2}_v,\pnorm{x_3}_v),
\end{equation}
où $v$ parcourt l'ensemble des places de $K$.

On fixe une place infinie $K\hookrightarrow \BmC$, de sorte que l'on a une inclusion $\BmP^2(K)\hookrightarrow \BmP^2(\BmC)$. Si l'extension $K/\BmQ$ est galoisienne, et que le point $P=[x_1:x_2:1]\in K^2 \subset \BmP^2(K)$ n'appartient pas à la droite à l'infini, on a:
\begin{equation}\label{minoration-height}
  h(P)\ge
  \frac{1}{[K:\BmQ]}
  \sum_{\sigma\in\MGal(K/\BmQ)}
  \log^+ \pnorm{P^\sigma},
\end{equation}
où $\pnorm{\cdot}$ est la norme euclidienne\footnote{$\pnorm{[x_1:x_2:1]}^2=x^2_1+x^2_2$} sur $\BmC^2$ et $\log^+(r):=\log \max(r,1)$ pour $r\in\BmR_+$.

\subsection{Hauteur na\"ive des points de Heegner.}
\label{sec:Heegner:height}
Lorsque $D$ est assez grand, il est clair que les points de Heegner $\varphi(z^\SmA_D)$ sont non nuls (cela découle du fait que le degré de $\varphi:X_\CmN\rightarrow E$ est borné quand $D\rightarrow -\infty$, tandis que le corps de définition de $z^\SmA_D$ est $H_D$ qui est de degré $2h(D)\rightarrow +\infty$). Dans une équation de Weierstrass, le seul point qui se situe sur la droite à l'infini est l'origine. On peut donc appliquer l'inégalité~\eqref{minoration-height} aux points de Heegner, ce qui donne avec~\eqref{NT-naive}:
\begin{equation}
  \widehat{h}(\varphi(z_D)) \ge 
  \frac{1}{h(D)}
  \sum_{\SmA\in \MCl_K} 
  \log^+ \pnorm{\varrho\circ \varphi(z^\SmA_D)} 
  + O_E(1).
\end{equation}
Reprenons quelques notations introduites au cours de la démonstration du Lemme~\ref{lem:singularitepointe}. Soit $\eta:\Gamma_0(N)\rightarrow \BmD$ l'application induite par $z\mapsto e^{2i\pi z}$, soit $A:=\{x+iy,\ 0\le x < 1, Y < y \le \infty\}$ et $B:=\eta(A)$. On choisit $Y$ suffisamment grand pour que $A$ soit inclus dans $X_0(N)(\BmC)$ et que $\eta|_A$ soit injective.

L'application composée $\varrho \circ \varphi \circ \eta|_A^{-1}:B\rightarrow \BmP^1(\BmC)$ est holomorphe et envoie $0$ sur le point $[0:1:0]$. En fait elle est étale au voisinage de $0$ de sorte que l'on a:
\begin{equation}
 \pnorm{\varrho \circ \varphi \circ \eta|_A^{-1}(q)} \gg_\varphi \pnorm{q}^{-1}, \qtext{pour tout $q \in B$.}
\end{equation}
On en déduit:
\begin{equation}
  \log \pnorm{\varrho\circ \varphi(z)} \ge 2\pi \cdot \Mim z + O_\varphi(1),
  \qtext{pour tout $z\in A$.}
\end{equation}
Avec l'asymptotique~\eqref{heegnerpointe} puis le Lemme~\ref{lem:LD:minoration} on obtient donc successivement:
\begin{equation}\label{minoration:heegner}
  \widehat{h}(\varphi(z_D)) \ge \frac{2\pi}{\Mvol(X_0(N))} \CmL_D +O_E(1) \gg_\varphi \log|D|.
\end{equation}

\begin{remark}\label{rem:pointes}
On pourrait améliorer le résultat précédent en considérant toutes les pointes de $X_0(N)$ dont l'image par $\varrho\circ \varphi$ est à l'infini. On peut vérifier que la contribution d'une telle pointe est identique à celle de la pointe $i\infty$ étudiée plus haut.
\end{remark}

\begin{remark}\label{rem:optimal}
L'estimée~\eqref{minoration:heegner} est proche du véritable ordre de grandeur. On montrera en effet dans~\cite{Temp:shifted} que $\widehat{h}(\varphi(z_D))$ est asymptotique à
 \begin{equation}
 24\frac{\deg(\varphi)}{[SL_2(\BmZ):\Gamma_0(N)]} \CmL_D
\end{equation}
lorsque $D\rightarrow -\infty$. Cette asymptotique est compatible avec la minoration~\eqref{minoration:heegner} et la Remarque~\ref{rem:pointes} puisque
\begin{equation}
 \pabs{%
 \{
 \text{pointes $\kappa$ tq $\varphi(\kappa)=0$}
 \}
 }\le \deg(\varphi).
\end{equation}
\end{remark}

\subsubsection*{Remerciements}
Je voudrais remercier Philippe Michel qui m'a proposé ce sujet de recherche. Je remercie également Akshay Venkatesh pour son invitation au Courant Institute en mai 2007 et toutes les observations qu'il a formulées, ainsi que Gérard Freixas, Amaury Thuillier et Thomas Vidick pour plusieurs discussions utiles.

%
%

%
%
%



\end{document}